\documentclass[11pt,a4paper]{article}

\usepackage{amsthm,amsmath,amsfonts,amssymb}
\usepackage[english]{babel}
\usepackage{graphicx}
\usepackage{verbatim}
\usepackage{color}
\usepackage{soul}

\newtheorem{proposition}{Proposition}
\newtheorem{theorem}{Theorem}
\newtheorem{lemma}{Lemma}
\newtheorem{corollary}{Corollary}

\DeclareMathOperator{\dist}{dist}
\def\Z{\ns Z}

\def\dist{{\rm dist}}
\def\x{\mbox{\boldmath $x$}}
\def\y{\mbox{\boldmath $y$}}

\def\vec0{\mbox{\boldmath $0$}}

\def\G{\Gamma}

\def\Z{\ns{Z}}

\def\G{\Gamma}

\def\Z{\mathbb Z}

\begin{document}

\title{From subKautz digraphs to cyclic Kautz digraphs
\thanks{This research is supported by \emph{MINECO} under project MTM2014-60127-P, and the {\em Catalan Research
Council} under project 2014SGR1147. This research has also received funding from the European Union's Horizon 2020 research and innovation programme under the Marie Sk\l{}odowska-Curie grant agreement No 734922.}}

\author{C. Dalf\'o
\\ \\
{\small Departament de Matem\`atiques} \\
{\small Universitat Polit\`ecnica de Catalunya} \\
{\small Barcelona, Catalonia} \\
{\small {\tt cristina.dalfo@upc.edu}}
}
\date{}
\maketitle

%\blfootnote{
%\begin{minipage}[l]{0.3\textwidth} \includegraphics[trim=10cm 6cm 10cm 5cm,clip,scale=0.15]{eu_logo} \end{minipage}  \hspace{-2cm} \begin{minipage}[l][1cm]{0.79\textwidth}
%   This research has also received funding from the European Union's Horizon 2020 research and innovation programme under the Marie Sk\l{}odowska-Curie grant agreement No 734922.
%  \end{minipage}}

\begin{abstract}
The Kautz digraphs $K(d,\ell)$ are a well-known family of dense digraphs, widely studied as a good model for interconnection networks.
Closely related to these, the cyclic Kautz digraphs $CK(d,\ell)$ were recently introduced by B\"ohmov\'a, Huemer and the author, and some of its distance-related parameters were fixed.
%\textcolor[rgb]{1.00,0.00,0.00}{\textst{In contrast with the Kautz digraphs, the set of vertices of the cyclic Kautz digraphs is invariant under
%cyclic permutations of the sequences representing them.}}
In this paper we propose a new approach to the cyclic Kautz digraphs
by introducing the family of the subKautz digraphs $sK(d,\ell)$, from where the cyclic Kautz digraphs can be obtained as line digraphs.
This allows us to give exact formulas for the distance between any two vertices of both $sK(d,\ell)$ and $CK(d,\ell)$. Moreover, we compute the diameter and the semigirth of both families, also providing  efficient routing algorithms to find the shortest path between any pair of vertices. Using these parameters, we also prove that $sK(d,\ell)$ and $CK(d,\ell)$ are maximally vertex-connected and super-edge-connected.
Whereas $K(d,\ell)$ are optimal with respect to the diameter, we show that $sK(d,\ell)$ and $CK(d,\ell)$ are optimal with respect to the mean distance,
whose exact values are given for both families when $\ell=3$. Finally, we provide a lower bound on the girth of $CK(d,\ell)$ and $sK(d,\ell)$.
%Finally, we give the exact values for the mean distances of both families.
%, proving that $sK(d,\ell)$ and $CK(d,\ell)$ have asymptotically optimal order with respect to the mean distance.
\end{abstract}

\noindent{\em Mathematics Subject Classifications:} 05C20, 05C50.

\noindent{\em Keywords:} Digraph, distance, diameter, mean distance, routing,  Kautz digraph, line digraph, (vertex-)connectivity, edge-connectivity, superconnectivity, semigirth, girth.

\section{Introduction}

Originally, the Kautz digraphs were introduced by Kautz~\cite{Ka68} in 1968. They have many applications, for example, they are useful as network topologies for connecting processors. The Kautz digraphs have the smallest diameter among all digraphs with their number of vertices and degree.

The cyclic Kautz digraphs $CK(d,\ell)$ were recently introduced by B\"ohmov\'a, Huemer and the author \cite{bdh15,bdh16}, as
subdigraphs with special symmetries of the Kautz digraphs $K(d,\ell)$, see for example Fiol, Yebra and Alegre \cite{fya84}. In contrast with these, the set of vertices of the cyclic Kautz digraphs is invariant under
cyclic permutations of the sequences representing them.
Thus, apart from their possible applications in interconnection networks, the cyclic Kautz digraphs $CK(d,\ell)$ could be relevant in coding theory, because they are related to cyclic codes. A linear code $C$ of length $\ell$ is called cyclic if, for every codeword
$c=(c_1,\ldots,c_\ell)$, the codeword $(c_\ell,c_1,\ldots,c_{\ell-1})$ is also in $C$.
This cyclic permutation allows to identify codewords with polynomials.
For more information about cyclic codes and coding theory, see Van Lint~\cite{vL92} (Chapter 6). With respect to other properties of the cyclic Kautz digraphs $CK(d,\ell)$, their number of vertices follows sequences that have several interpretations. For example, for $d=2$ (that is, 3 different symbols) and $\ell=2,3,\ldots$, the number of vertices follows the sequence $6,6,18,30,66,\ldots$ According to the On-Line Encyclopedia of Integer Sequences~\cite{Sl}, this is the sequence A092297.
For $d=3$ (4 different symbols) and $\ell=2,3,\ldots$, we get the sequence $12,24,84,240,732,\ldots$ corresponding to
A226493 and A218034 in~\cite{Sl}.

In this paper we give an alternative definition of $CK(d,\ell)$, by introducing the family of the subKautz digraphs $sK(d,\ell)$, from where the cyclic Kautz digraphs can be obtained as line digraphs. We present the exact formula of the distance between any two vertices of $sK(d,\ell)$ and $CK(d,\ell)$. This allows us to compute the diameter and the semigirth of both families, also providing an efficient routing algorithm to find the shortest path between any pair of vertices. Using these parameters, we also prove that $sK(d,\ell)$ and $CK(d,\ell)$ are maximally vertex-connected and super-edge-connected. Whereas $K(d,\ell)$ are optimal with respect to the diameter, we show that $sK(d,\ell)$ and $CK(d,\ell)$ are optimal with respect to the mean distance, whose exact values are given for both families when $\ell=3$. Finally, we provide a lower bound on the girth of $sK(d,\ell)$ and $CK(d,\ell)$.

\subsection{Notation}

We  consider  simple digraphs (or directed graphs) without loops or multiple arcs,
and we follow the usual notation for them. That is, a \emph{digraph} $G=(V,E)$ consists of a (finite) set $V=V(G)$ of vertices and a set $E=E(G)$ of arcs (directed edges) between vertices of $G$.
%There are no
%\emph{multiple arcs}, that is, there is at most one arc from each vertex to any other.
If $a=(u,v)$ is an arc between vertices $u$ and $v$, then the vertex $u$
is \emph{adjacent to} the vertex $v$, and the vertex $v$
is \emph{adjacent from} $u$. Let $\Gamma^+(v)$ and $\Gamma^-(v)$ denote the set of vertices adjacent from and to the vertex $v$, respectively. Their cardinalities are the \emph{out-degree} $\delta^+(v)=|\Gamma^+(v)|$ of the vertex $v$, and the \emph{in-degree} $\delta^-(v)=|\Gamma^-(v)|$ of the vertex $v$. A digraph $G$ is called $d$\emph{-out-regular} if $\delta^+(v)=d$, $d$\emph{-in-regular} if $\delta^-(v)=d$, and $d$\emph{-regular} if $\delta^+(v)=\delta^-(v)=d$, for all $v\in V$. The minimum degree $\delta=\delta(G)$ of $G$ is the minimum over all the in-degrees and out-degrees of the vertices of $G$. A \emph{digon} is a directed cycle on 2 vertices.
For other notation, and unless otherwise
stated, we follow the book by Bang-Jensen and Gutin~\cite{BG07}.

In the {\em line digraph} $L(G)$ of a digraph $G$, each vertex represents an arc of $G$,
$V(L(G))=\{uv: (u,v)\in E(G)\}$, and a vertex $uv$ is adjacent to a vertex $wz$ when $v=w$, that is, when in $G$ the arc $(u,v)$ is adjacent to the arc $(w,z)$: $u\rightarrow v(=w)\rightarrow z$.
Fiol and Llad\'{o} defined in~\cite{FiLl92} the partial line digraph $PL(G)$ of a digraph $G$, where some (but not necessarily all, as in the line digraph $L(G)$) of the arcs in $G$ become vertices in $PL(G)$.
Let $E'\subseteq E$ be a subset of arcs which are incident to all vertices of $G$, that is, $\{v:(u,v)\in E'\}=V$. A digraph $PL(G)$ is said to be a \emph{partial line digraph} of $G$ if its vertices
represent the arcs of $E'$, that is, $V(PL(G))=\{uv:(u,v)\in E'\}$, and a vertex $uv$ is adjacent to the vertices $v'w$, for each $w\in\Gamma_G^+(v)$,
where
%$$
%v'=\left\{
%\begin{array}{ll}
%     v & \mbox{if } vw\in V(PL(G)), \label{cas1}\\
%     \mbox{any other vertex of } \Gamma_G^-(w) & \\[-0.1cm]
%     \mbox{such that } v'w\in V(PL(G)) & \mbox{otherwise.}
%     \label{cas2}
%   \end{array}
%\right.
%$$
$$
v'=\left\{
\begin{array}{ll}
     v & \mbox{if } vw\in V(PL(G)), \label{cas1}\\
     \mbox{any other vertex of } \Gamma_G^-(w) \mbox{ such that } v'w\in V(PL(G))& \mbox{otherwise}\label{cas2}.
   \end{array}
\right.
$$

\vskip-.25cm

A digraph $G$ is \emph{strongly connected} when, for any pair of vertices $x,y\in V$, there always exists an $x\rightarrow y$ path, that is, a path from the vertex $x$ to the vertex $y$. The \emph{strong connectivity} $\kappa=\kappa(G)$ (or strong vertex-connectivity) of $G$ is the smallest number of vertices whose deletion results in a digraph that is either not strongly connected or trivial. Analogously, the \emph{strong arc-connectivity} $\lambda=\lambda(G)$ of $G$ is the smallest number of arcs whose deletion results in a not strongly connected digraph. Since we only deal with strong connectivities, from now on we are going to refer to them simply as connectivities. Now we only consider connected digraphs, so $\delta\geq1$. It is known that
$\kappa\leq\lambda\leq\delta$,
see Geller and Harary~\cite{GeHa70}. A digraph $G$ is \emph{maximally connected} when $\kappa=\lambda=\delta$.

If $G$ is a \emph{maximally arc-connected} digraph $(\lambda=\delta)$, then any set of arcs adjacent from [to] a vertex $x$ with out-degree [in-degree] $\delta$ is a minimum order arc-disconnecting set. Similarly, if $G$ is a \emph{maximally vertex-connected} digraph $(\kappa=\delta)$, the set of vertices adjacent from [to] $x$ is a minimum order vertex-disconnecting set. In this context, these arc or vertex sets are called \emph{trivial}. Note that the deletion of any trivial set isolates a vertex of in-degree or out-degree $\delta$. A digraph $G$ is \emph{super-}$\kappa$ if every minimum vertex-disconnecting set is trivial. Analogously, $G$ is \emph{super-}$\lambda$ is all its minimum arc-disconnecting sets are trivial. If $G$ is super-$\kappa$, then $\kappa=\delta$, and if $G$ is super-$\lambda$, then $\lambda=\delta$. In general, the converses are not true.

We say that a digraph is \emph{weakly antipodal} when every vertex $u$ has exactly one vertex $v$ at maximum distance (the diameter), and it is
\emph{antipodal} when simultaneously $u$ and $v$ are at maximum distance from each other. For instance, the directed cycle $C_n$ is weakly antipodal, whereas the symmetric directed cycle $C_n^*$ with even $n$ is antipodal.

\subsection{The semigirth}

%Recall that a digraph is $k$-\emph{reachable} if for every pair of (not necessarily different) vertices $x,y\in V$ there exists
%a path of exactly $k$ arcs from $x$ to $y$. See Comellas, Fiol, and G\'{o}mez~\cite{CoFiGo94}, and Comellas and Fiol~\cite{CoFi95}.

We recall the definition of the \emph{semigirth}: For a given digraph $G$, let $\gamma=\gamma(G)$, for $1\leq \gamma\leq D$, where $D$ is the diameter, be the greatest integer
  such that for any two (not necessarily different) vertices $x,y\in V$,\\
\indent$(a)$ if $\dist(x,y)<\gamma$, then the shortest $x\rightarrow y$ path is unique, and there is no an $x\rightarrow y$ path of length $\dist(x,y)+1$;\\
\indent$(b)$ if $\dist(x,y)=\gamma$, then there is only one shortest $x\rightarrow y$ path.

Note that $\gamma$ is well defined when $G$ has no loops. In ~\cite{FaFi89}, F\`{a}brega and Fiol proved that, if a digraph $G$ (different from a directed cycle) has semigirth $\gamma$, then its line digraph $L(G)$ has semigirth $\gamma+1$. The diameter %and the $k$-reachability
also has the same behaviour, that is, if the diameter of $G$ is $D$, then its line digraph $L(G)$ has diameter $D+1$.

We also recall two results from F\`{a}brega and Fiol~\cite{FaFi89} on the connectivities and superconnectivities.

\begin{theorem}[\cite{FaFi89}]
\label{theo1FaFi}
Let $G=(V,E)$ be a loopless digraph with minimum degree $\delta>1$, semigirth $\gamma$, diameter $D$ and connectivities $\lambda$ and $\kappa$.\\
\indent$(a)$ If $D\leq2\gamma$, then $\lambda=\delta$.\\
\indent$(b)$ If $D\leq2\gamma-1$, then $\kappa=\delta$.
\end{theorem}

\begin{theorem}[\cite{FaFi89}]
\label{theo2FaFi}
Let $G=(V,E)$ be a loopless digraph with minimum degree $\delta\geq3$, semigirth $\gamma$, and diameter $D$.\\
\indent$(a)$ If $D\leq2\gamma$, then $G$ is super-$\lambda$.\\
\indent$(b)$ If $D\leq2\gamma-2$, then $G$ is super-$\kappa$.
\end{theorem}

\subsection{Moore digraphs with respect to the diameter and the mean distance}
\label{section:mean-dist}

The Moore bound on the number of vertices for digraphs with diameter $D$ and maximum degree $\Delta$ is
$N(\Delta,D)=\frac{\Delta^{D+1}-1}{\Delta-1}$ for $\Delta>1$ and $N(1,D)=D+1$.
Notice that $N\sim O(\Delta^{D})$.

The digraphs that attain the Moore bound $N(\Delta,D)$ are called Moore digraphs. The only Moore digraphs are the directed cycles on $D+1$ vertices and the complete digraphs on $\Delta+1$ vertices.
For $D>1$ and $\Delta>1$, there are no Moore digraphs. For more information, see the survey by Miller and \v{S}ira\v{n}~\cite{MiSi13}.

%In an analogous way as (\ref{Moore-bound}) is the Moore bound on the number of vertices,
%Besides of the Moore bound on the number of vertices, we also can give a Moore-like bound on the mean distance, as we show in the next result.
The mean distance corresponding to a digraph attaining the Moore bound is given in the following result. As the only Moore digraphs are the directed cycles and the complete digraphs, this bound gives an idea of how close is a digraph (with diameter $D$ and maximum degree $\Delta$) of being a Moore digraph.

\begin{lemma}
The mean distance $\overline{\partial}(\Delta,D)$ of a digraph with diameter $D$ and maximum degree $\Delta$ attaining the Moore bound would be
$$
\overline{\partial}(\Delta,D)=\frac{D\Delta^{D+2}-(1+D)\Delta^{D+1}+\Delta}{\Delta^{D+2}-\Delta^{D+1}-\Delta+1}.
$$
\end{lemma}

\begin{proof}
We compute $\overline{\partial}(\Delta,D)$ taking into account that the maximum number of vertices at distance $k$ is $\Delta^k$.
\begin{eqnarray*}
\overline{\partial}(\Delta,D) &=& \frac{1}{N(\Delta,D)}\sum_{k=0}^D k\Delta^k=\frac{\Delta}{N(\Delta,D)}\sum_{k=0}^D k\Delta^{k-1}= \frac{\Delta}{N(\Delta,D)}\left(\sum_{k=0}^D \Delta^{k}\right)'\\
                              &=& \frac{\Delta}{N(\Delta,D)} \left(\frac{\Delta^{D+1}-1}{\Delta-1}\right)'= \frac{D\Delta^{D+2}-(1+D)\Delta^{D+1}+\Delta}{\Delta^{D+2}-\Delta^{D+1}-\Delta+1}.
\end{eqnarray*}
\vskip-.75cm
\end{proof}

\vskip-.5cm

We can define a digraph as optimal with respect to the diameter (the maximum delay in a message transmission), but also with respect to the mean distance (the average delay in a message transmission). So, we can say that a digraph is optimal when, if $N$ is of the order of $\Delta^k$, then its mean distance is of the order of $k$, that is, when $\overline{\partial}\sim O(\log_{\Delta}N)$.

%això seria el cas dels digrafs amb ordre "proper" a la fita de Moore

\section{Kautz-like digraphs}

The Kautz $K(d,\ell)$, the subKautz $sK(d,\ell)$, the cyclic Kautz $CK(d,\ell)$, and the modified cyclic Kautz $MCK(d,\ell)$ digraphs have vertices represented by words on an alphabet, and adjacencies between vertices correspond to shifts of the words. In these Kautz-like digraphs a path $\x\rightarrow\y$ corresponds to a sequence beginning with $\x=x_1x_2\ldots x_{\ell}$ and finishing with $\y=y_1y_2\ldots y_{\ell}$, where every subsequence of length $\ell$ corresponds to a vertex of the corresponding digraph.

\subsection{Kautz and subKautz digraphs}

%\begin{figure}[t]
%\vskip-.75cm
%    \begin{center}
%  \includegraphics[width=12cm]{Kautz-normals-sense-color.pdf}
%    \vskip-3cm
%  \includegraphics[width=12cm]{subKautz-normals-sense-color.pdf}
%  \end{center}
%  \vskip-4cm
%  \caption{Some examples of Kautz and subKautz digraphs.} \label{fig:Kautz}
%\end{figure}

\begin{figure}[t]
%\vskip-.75cm
    \begin{center}
  \includegraphics[width=14cm]{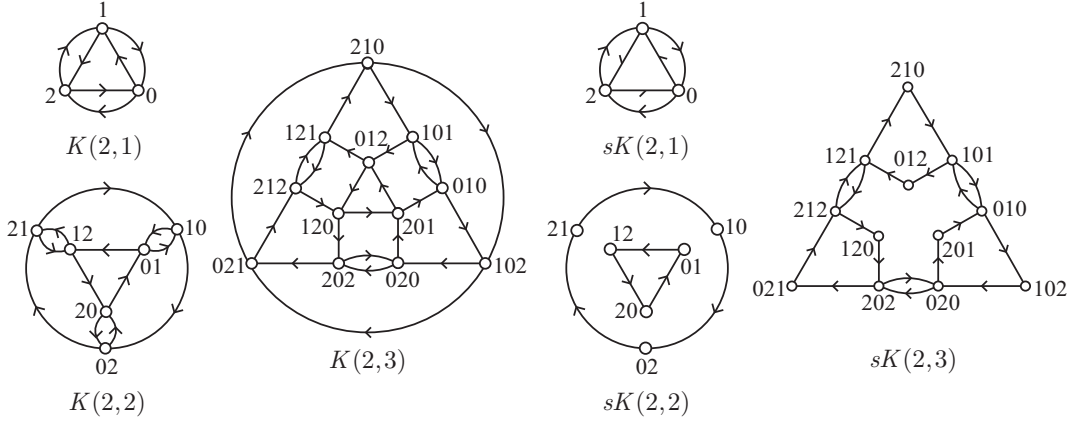}
  %\vskip-4.5cm
  \caption{Some examples of the Kautz and the subKautz digraphs.} \label{fig:Kautz}
  \end{center}
\end{figure}

Next, we recall the definitions of the Kautz $K(d,\ell)$, and we define a new family of Kautz-like digraphs called the subKautz digraphs $sK(d,\ell)$. See examples of both in Figure~\ref{fig:Kautz}.

A \emph{Kautz digraph} $K(d,\ell)$ has the vertices $x_1x_2\ldots x_{\ell}$, where $x_i\in \mathbb{Z}_{d+1}$, with $x_{i}\neq x_{i+1}$ for $i=1,\ldots,\ell-1$, and adjacencies
$$
x_1x_2\ldots x_{\ell}\quad\rightarrow\quad x_2x_3\ldots x_{\ell} y,\qquad y\neq x_{\ell}.
$$

Given integers $d$ and $\ell$, with $d,\ell\geq2$, a \emph{subKautz digraph} $sK(d,\ell)$ has set of vertices
$V=\{x_1x_2\ldots x_{\ell}:x_i\neq x_{i+1},\ i=1,\ldots,\ell-1\}\subset \Z_{d+1}^{\ell}$,
and adjacencies
\begin{equation}
\label{eq:adj-sK}
x_1x_2\ldots x_{\ell}\quad\rightarrow\quad x_2\ldots x_{\ell}x_{\ell+1}, \qquad x_{\ell+1}\neq x_1,x_{\ell}.
\end{equation}
Hence, the subKautz digraph $sK(d,\ell)$ has $d^{\ell}+d^{\ell-1}$ vertices, as the Kautz digraph $K(d,\ell)$. Besides, the out-degree of
a vertex $x_1x_2\ldots x_{\ell}$ is $d$ if $x_1=x_{\ell}$, and $d-1$ otherwise.
In particular, the subKautz digraph $sK(d,2)$ is $(d-1)$-regular and can be obtained from the Kautz digraph $K(d,2)$ by removing all its arcs forming a digon.
%As an example of a subKautz digraph, $sK(3,3)$ is shown in Figure~\ref{fig:sub_K(3,3)}.

Note that the subKautz digraph $sK(d,\ell)$ is a subdigraph of the Kautz digraph $K(d,\ell)$.

\subsection{Cyclic Kautz and modified cyclic Kautz digraphs}

%\begin{figure}[t]
%\vskip-.75cm
%    \begin{center}
%  \includegraphics[width=8cm]{donut.pdf}
%  \vskip-1cm
%  \includegraphics[width=8cm]{donut-modificat.pdf}
%  \vskip-.25cm
%  \caption{Some examples of cyclic Kautz and modified cyclic Kautz digraphs.} \label{fig:Kautz-ciclics}
%  \end{center}
%\end{figure}

\begin{figure}[t]
%\vskip-.75cm
    \begin{center}
  \includegraphics[width=12cm]{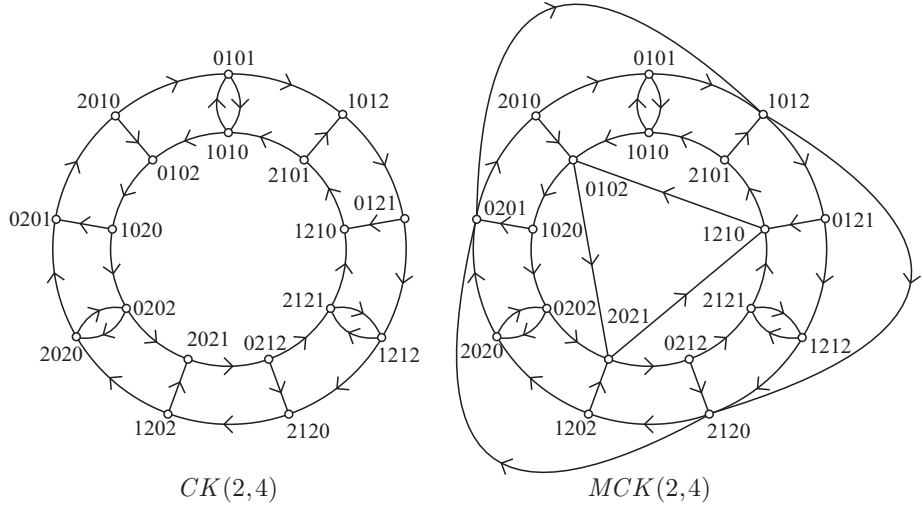}
  %\vskip-1.75cm
  \caption{An example of a cyclic Kautz digraph and a modified cyclic Kautz digraph.}
  \label{fig:Kautz-ciclics}
  \end{center}
\end{figure}

Next, we recall the definitions of the cyclic Kautz digraphs $CK(d,\ell)$ and the modified cyclic Kautz digraphs $MCK(d,\ell)$. See an example of both in Figure~\ref{fig:Kautz-ciclics}.

A \emph{cyclic Kautz digraph} $CK(d,\ell)$ has the vertices $x_1x_2\ldots x_{\ell}$, where $x_i\in \mathbb{Z}_{d+1}$, with $x_{i}\neq x_{i+1}$ for $i=1,\ldots,\ell-1$, and $x_{\ell}\neq x_1$, and adjacencies
$$
x_1x_2\ldots x_{\ell}\quad\rightarrow\quad x_2x_3\ldots x_{\ell} y,\qquad y\neq x_2,x_{\ell}.
$$

Note that the cyclic Kautz digraphs $CK(d,\ell)$ are subdigraphs of the Kautz digraph $K(d,\ell)$. It was proved in \cite{bdh16} that
when $d=2$ the cyclic Kautz digraphs $CK(2,\ell)$ are not connected (except for the case $\ell=4$), and when $\ell=2$ the cyclic Kautz digraphs $CK(d,2)$ coincide with the Kautz digraphs $K(d,2)$.

Recall that the diameter of the Kautz digraphs is optimal, that is, for a fixed out-degree $d$ and number of vertices $(d+1)d^{\ell-1}$, the Kautz digraph $K(d,\ell)$ has the smallest diameter $(D=\ell)$ among all digraphs with $(d+1) d^{\ell-1}$ vertices and degree $d$ (see, for example, Miller and \v{S}ir\'{a}\v{n}~\cite{MiSi13}).
Since the diameter of the cyclic Kautz digraphs $CK(d,\ell)$ is greater than the diameter of the Kautz digraphs $K(d,\ell)$, in~\cite{bdh16b} we constructed the \emph{modified cyclic Kautz digraphs} $MCK(d,\ell)$ by adding some arcs to $CK(d,\ell)$, in order to obtain the same diameter as $K(d,\ell)$, without increasing the maximum degree.
In a cyclic Kautz digraph $CK(d,\ell)$, a vertex labeled with $a_2\ldots a_{\ell+1}$ is forbidden if $a_2=a_{\ell+1}$. For each label, we replace the first symbol $a_2$ by one of the possible symbols $a_2'$ such that now $a_2'\neq a_3, a_{\ell+1}$ (so $a_2'\ldots a_{\ell+1}$ represents a vertex). Then, we add arcs from the vertex $a_1\ldots a_{\ell}$ to the vertex $a_2'\ldots a_{\ell+1}$, with $a_1\neq a_{\ell}$ and $a_2'\neq a_3,a_{\ell+1}$. Note that $CK(d,\ell)$ and $MCK(d,\ell)$ have the same vertices, because we only add arcs to $CK(d,\ell)$ to obtain $MCK(d,\ell)$.

\begin{lemma}
\indent$(a)$ The cyclic Kautz digraph $CK(d,\ell)$ is the line digraph of the subKautz digraph $sK(d,\ell-1)$, that is, $CK(d,\ell)=L(sK(d,\ell-1))$.\\
\indent$(b)$ The modified cyclic Kautz digraph $MCK(d,\ell)$ is the partial line digraph of the Kautz digraph $K(d,\ell-1)$, that is, $MCK(d,\ell)=PL(K(d,\ell-1))$.
\end{lemma}

\begin{proof}
\indent$(a)$ From (\ref{eq:adj-sK}) we can write the arcs $(x_1x_2\ldots x_{\ell-1},x_2\ldots x_{\ell-1}x_{\ell})$ of $sK(d,\ell-1)$ as
$x_1x_2\ldots x_{\ell-1}x_{\ell}$ with $x_i\neq x_{i+1}$ and $x_1\neq x_{\ell}$,
which corresponds to the vertices of $CK(d,\ell)$. Moreover, two arcs are adjacent in $sK(d,\ell-1)$ if
  $$
    x_1x_2\ldots x_{\ell} \quad\rightarrow\quad x_2\ldots x_{\ell}x_{\ell+1},
  $$
  where $x_1\neq x_{\ell}$, as required for the vertices of $CK(d,\ell)$.\\
\indent$(b)$ This was proved in \cite{bdh16b}. In taking the partial line digraph, it suffices to consider only the arcs in
  $K(d,\ell-1)$ that are also in $sK(d,\ell-1)$.
\end{proof}

By using spectral techniques, %B\"ohmov\'a, Dalf\'o, and Huemer
the order $n_{d,\ell}$ of a cyclic Kautz digraph $CK(d,\ell)$ was given in~\cite{bdh15,bdh16}. Here we use
a
%n alternative
combinatorial proof of this result.

\begin{proposition}
\label{propo0}
The order $n_{d,\ell}$ of a cyclic Kautz digraph $CK(d,\ell)$ (that coincide with the size of the subKautz digraph $sK(d,\ell-1)$) is $n_{d,1}=d+1$ and
\begin{equation}
\label{orderCK}
n_{d,\ell}=d^{\ell}+(-1)^{\ell}d\qquad \mbox{for $\ell\ge 2$}.
\end{equation}
\end{proposition}

\begin{proof}
The number $N_{d,\ell}$ of sequences $x_1x_2\ldots x_{\ell}$ with $x_i\neq x_{i+1}$ for $i=1,\ldots,\ell-1$
(vertices of $K(d,\ell)$) is $d^{\ell}+d^{\ell-1}$.
Then,  to compute $n_{d,\ell}$, % =N_{d,\ell}-n'_{d,\ell}$, to such a number
we must subtract from $N_{d,\ell}$ the number $n'_{d,\ell}$ of sequences $x_1x_2\ldots x_{\ell}$
such that $x_1=x_{\ell}$. But this is the same as the number of sequences $x_2\ldots x_{\ell}$ with $x_2\neq x_{\ell}$ and $x_i\neq x_{i+1}$ for $i=2,\ldots,\ell-1$, which is $n_{d,\ell-1}$. Consequently, we get the recurrence
\begin{equation}
\label{recur-n}
n_{d,\ell}=d^{\ell}+d^{\ell-1}-n_{d,\ell-1}\qquad \mbox{for $\ell\ge 3$}.
\end{equation}
Thus, \eqref{orderCK} follows by applying recursively \eqref{recur-n} and using that $n_{d,2}=d^2+d$.
\end{proof}

In the following result we prove a way of finding an $sK(d,\ell)$ a from the Kautz digraphs $K(d,\ell)$. We use the
cyclic Kautz digraphs $CK(d,\ell)$ in the proof.

\begin{lemma}
The subKautz digraphs $sK(d,\ell)$ can be obtained from the Kautz digraphs $K(d,\ell)$ by removing all the arcs
%of the directed cycles of length even if $\ell$ is even and of length odd if $\ell$ is odd, up to $\ell$ in both cases.
of the closed walks of length $\ell$ in the complete symmetric digraph $K^*_{d+1}$.
\end{lemma}

\begin{proof}
From their definition, the subKautz digraphs $sK(d,\ell)$ are obtained from $K(d,\ell)$ by removing the arcs of the form
$x_1x_2\ldots x_{\ell}\rightarrow x_2\ldots x_{\ell}x_{1},$
which correspond to the vertices $x_1x_2\ldots x_{\ell}x_{1}$ of $K(d,\ell+1)$,
which in turn correspond to the closed walks of length  $\ell$ in the complete symmetric digraph $K^*_{d+1}$.
%The Kautz digraphs $K(d,\ell)$ have $d^{\ell}+d^{\ell-1}$ vertices and $d^{\ell+1}+d^{\ell}$ arcs. The number of vertices of
%$sK(d,\ell)$ also  is $d^{\ell}+d^{\ell-1}$, and we compute its number of arcs as follows.
%Since $L(K(d,\ell))=K(d,\ell+1)$ and $L(sK(d,\ell))=CK(d,\ell+1)$,
%to obtain the number of arcs of $sK(d,\ell)$, we subtract to the number of arcs of $K(d,\ell)$ the number of vertices of
%$CK(d,\ell)$ (as in Proposition~\ref{propo0}), that is, the size of $sK(d,\ell)$ is
%$d^{\ell+1}+d^{\ell}-(d^{\ell}+(-1)^{\ell}d)=d^{\ell+1}+(-1)^{\ell+1}$,
%which coincide with the order of $CK(d,\ell+1)$.
\end{proof}

A simple property of symmetry shared by all the Kautz-like digraphs is the following. The converse digraph is obtained by changing the direction of all the arcs in the original digraph.

\begin{lemma}
The Kautz digraphs $K(d,\ell)$, the subKautz digraphs $sK(d,\ell)$, and the cyclic Kautz digraphs $CK(d,\ell)$ are isomorphic to their converses.
\end{lemma}

\begin{proof}
Since  the mapping $\Psi(x_1x_2\ldots x_{\ell})=x_{\ell}\ldots x_2x_1$ satisfies
\begin{eqnarray*}
  \Psi(\G^+(\{x_1x_2\ldots x_{\ell}\}))&=&\Psi(\{x_2x_3\ldots x_{\ell}y\ : y\in\Z_{d+1}, y\neq x_{\ell}\})\\
  &=& \{y x_{\ell}\ldots x_3x_2: y\in\Z_{d+1}, y\neq x_{\ell}\}\\
  &=& \G^- (\{x_{\ell}\ldots x_2x_1\}) = \G^- (\Psi(\{x_1x_2\ldots x_{\ell}\})),
\end{eqnarray*}
%\begin{align*}
%\Psi(\G^+(x_1x_2\ldots x_{\ell})) & =\Psi(x_2x_3\ldots x_{\ell}y)=y x_{\ell}\ldots x_3x_2\\
% &=\G^- (x_{\ell}\ldots x_2x_1)=\G^- (\Psi(x_1x_2\ldots x_{\ell})),
% \end{align*}
where in the case of $CK(d,\ell)$ also $y\neq x_{2}$,
 it is an isomorphism between every of such digraphs and its converse.
\end{proof}

\section{Routing, distances and girth in $CK(d,\ell)$}

In this section, we only need to consider the cases with $d\geq3$ and $\ell\geq3$ because, as said in the Introduction, when $d=2$ the cyclic Kautz digraphs $CK(2,\ell)$ are not connected (except for the case $\ell=4$), and when $\ell=2$,
the cyclic Kautz digraphs $CK(d,2)$ coincide with the Kautz digraphs $K(d,2)$.

We begin the study of the routing and distance in $CK(d,\ell)$ with the case $d,\ell\ge 4$ and, afterwards, we deal with the case $d=3$ or $\ell=3$.

%\subsection{The case $d=2$}
%
%Let us begin with the case $d=2$. The computation of the numbers of vertices in the strong components of
%$CK(2,\ell)$ can be carried out by using the concept of ``imprint'' introduced by B\"ohmov\'a, Dalf\'o, and Huemer~\cite{bdh16}. Namely, they proved that all vertices in a given component must have the same imprint. This gives \ldots
%
%For instance, some specific values are:
%\begin{itemize}
%\item
%$\ell=2$: 1 component with 6 vertices (Kautz digraph).
%\item
%$\ell=3$: 2 components with 3 \& 3 vertices (see the figure in \cite{bdh16}).
%\item
%$\ell=4$: 1 component with 18 vertices (see the figure in  \cite{bdh16}).
%\item
%$\ell=5$: 2 components with 15 \& 15 vertices.
%\item
%$\ell=6$: 3 components with 3 \& 60 \& 3 vertices (the component with 3 vertices are cycles).
%\item
%$\ell=7$:  2 components with 63 \& 63 vertices.
%\item
%$\ell=8$:  3 components with 24 \& 210 \& 24 vertices.
%\item[\vdots]
%\end{itemize}
%Then, if we add up all these numbers of vertices, we see that $N=2^\ell+2(-1)^\ell$, according to Proposition \ref{propo0}.

\subsection{Routing and distances when $d,\ell\ge 4$}

For simplicity, and without loss of generality, we fix the length $\ell$ of the sequences, for instance, assume that we are dealing with the cyclic Kautz digraph $CK(d,7)$ on the alphabet $\Z_{d+1}=\{0,1,\ldots,d\}$ with $d\ge 4$.

Let us consider two generic vertices:
$$
\begin{array}{ccccccccc}
\x &=& x_1&x_2&x_3&x_4&x_5&x_6&x_7, \\
\y &=& y_1&y_2&y_3&y_4&y_5&y_6&y_7,
\end{array}
$$
and the {\em extended sequence}  of $\x$, that is,
$$
\begin{array}{ccccccccccccccc}
\tilde{\x}& = & x_1&x_2&x_3&x_4&x_5&x_6&x_7&\overline{x_2}&\overline{x_3}&
\overline{x_4}&
\overline{x_5}&\overline{x_6}&\overline{x_7},
\end{array}
$$
where $\overline{x_i}\in\Z_{d+1}$ and $\overline{x_i}\neq x_i$.
(Note that we also can interpret $\tilde{\x}$ as a set of sequences of length $2\ell-1$.)
Then, to find the distance $\dist(\x,\y)$, we compute the intersection $\tilde{\x}\sqcap \y$, which is the maximum subsequence of $\tilde{\x}$ that coincides with the initial subsequence of $\y$. Analogously, the intersection $\x\sqcap \y$ is the maximum final subsequence of $\x$ that coincides with the initial subsequence of $\y$.
According to the length of such a subsequence, we distinguish three cases:

\begin{enumerate}
\item[$(a)$]
$|\tilde{\x}\sqcap \y|> \ell-1$ ($\Rightarrow \ell-1 \ge |\x\sqcap \y|\ge 1$):\\
For instance, suppose that $|\x\sqcap \y|=4$, so that we have the coincidence pattern:
$$
\begin{array}{ccccccccccccc}
x_1&x_2&x_3&x_4&x_5&x_6&x_7&\overline{x_2}&\overline{x_3}&
\overline{x_4}&
\overline{x_5}&\overline{x_6}&\overline{x_7}\\
 & & & y_1&y_2&y_3&y_4&y_5&y_6&
y_7&
 & &
\end{array}
$$
where $y_i=x_{i+3}$ for $i=1,\ldots,4$, and\\
\indent$(a1)$ $y_5\neq x_2$ and $y_5\neq y_4=x_7$,\\
\indent$(a2)$ $y_6\neq x_3,y_5$,\\
\indent$(a3)$ $y_7\neq x_4=y_1$ and $y_7\neq y_6$.\\
Then, the only shortest path from $\x$ to $\y$ is
$$
\x\!=\!x_1x_2x_3y_1y_2y_3y_4\ \rightarrow\ x_2x_3y_1y_2y_3y_4y_5\ \rightarrow\ x_3y_1y_2y_3y_4y_5y_6\ \rightarrow\ y_1y_2y_3y_4y_5y_6y_7\!=\!\y.
$$
Hence, in this case, $\dist(\x,\y)=3$ and, in general,
$$
\dist(\x,\y)=\ell-|\x\sqcap \y|\le \ell-1.
$$
%To prove that the semigirth $\gamma$ is at least $\ell-1$, we check that in this situation there is no $\x\rightarrow\y$ path of length $\dist(x,y)+1$.
%Indeed, if $y_1=x_4$, then $y_1\neq x_5$ as $x_4\neq x_5$. So, there is no path of length $\ell$, and $\gamma\geq\ell-1$.

\item[$(b)$] $|\tilde{\x}\sqcap \y|=\ell- 1$:\\
If $y_1\neq x_7$, we reason as in case $(a)$ and we get $\dist(\x,\y)= \ell$. Otherwise, if $y_1=x_7$, the sequence $x_2x_3\ldots x_7y_1$ does not correspond to any vertex. Then, we have to consider the `second largest' intersection satisfying the next case $(c)$: $1\leq|\tilde{\x}\sqcap \y|<\ell-1$. (Since $\ell\ge 4$, we prove later that this is always possible.) Thus, we get $\dist(\x,\y)=2\ell-1-|\tilde{\x}\sqcap \y|$.

%To prove that the semigirth $\gamma$ is $\ell$, we check that in this situation there is no $\x\rightarrow\y$ path of length $\dist(x,y)+1$.
%In this case, it is clear that the only $\x\rightarrow\y$ path is
%$$
%\x=x_1x_2x_3x_4x_5x_6x_7\ \rightarrow\
%   x_2x_3x_4x_5x_6x_7y_1\ \rightarrow\  \cdots\ \rightarrow\
%   y_1y_2y_3y_4y_5y_6y_7=\y,
%$$
%and it has length $\ell$.

Note that the number of vertices at distance $\ell$ is of the order of $d^{\ell}$, which also corresponds to the optimal mean distance.

\item[$(c)$] $1\leq|\tilde{\x}\sqcap \y|<\ell-1$:\\
Suppose, for instance, that $|\widetilde{\x}\sqcap \y|=3$.
$$
\begin{array}{ccccccccccccccccc}
x_1 & x_2 & x_3 & x_4 & x_5 & x_6 & x_7 & \overline{x_2} & \overline{x_3} &
\overline{x_4} & \overline{x_5} & \overline{x_6} & \overline{x_7} & & & & \\
 & & & & & &  &\overline{y_4} &\overline{y_5} &\overline{y_6} & y_1 & y_2 & y_3 & y_4 & y_5 & y_6 & y_7 \\
  & & & & & & = &z_1 &z_2 &z_3& y_1 & y_2 & y_3 & y_4 & y_5 & y_6 & y_7
\end{array}
$$
where\\
\indent$(c1)$ $z_1\neq x_7,x_2,y_4$,\\
\indent$(c2)$ $z_2\neq z_1,x_3,y_5$,\\
\indent$(c3)$ $z_3\neq z_2,x_4,y_6,y_1$,\\
\indent$(c4)$ $y_1\neq z_3,x_5$.\\
Then, $\dist(\x,\y)=10$ and, in general,
$$
\dist(\x,\y)=2\ell-1-|\tilde{\x}\sqcap \y|\le 2\ell-2.
$$
\end{enumerate}

Now we are ready to prove the following result.

\begin{theorem}
\label{theo1}
The diameter of the cyclic Kautz digraph $CK(d,\ell)$ with $d,\ell\ge 4$ is $D=2\ell-2$.
\end{theorem}

\begin{proof}
First, we claim that $|\tilde{\x}\sqcap \y|\ge 1$. Indeed, on the contrary,  we would have that $y_1=x_7\neq x_6$ and $y_2\neq y_1=x_7$. Consequently,
$|\tilde{\x}\sqcap \y|\ge 2$, a contradiction.
Then, if $|\tilde{\x}\sqcap \y|= 1$, we are in case $(c)$. Otherwise, from the above reasoning, we  have at least an intersection $|\tilde{\x}\sqcap \y|=2<\ell-1$, as $\ell\ge 4$, and case $(c)$ applies again.
%According to the above results, we first need to prove that $|\tilde{\x}\sqcap \y|\ge 1$. Indeed, on the contrary,  we would have that $y_1=x_7\neq x_6$ and $y_2\neq y_1=x_7$. Consequently,
%$|\tilde{\x}\sqcap \y|\ge 2$, a contradiction.

%Then if, in one step, we go from $\x$ to $\x'=x_2x_3x_4x_5x_6x_7z_1$, where $z_1\neq x_7,x_2,y_4$, we have now the coincidence pattern
%$$
%\begin{array}{cccccccccccccccc}
%x_2 & x_3 & x_4 & x_5 & x_6 & x_7 & z_1 & \overline{x_3} & \overline{x_4} & \overline{x_5} & \overline{x_6} & \overline{x_7} & \overline{z_1} &  &  & \\
% &  &  &  &  &  &  & \overline{y_5} & \overline{y_6} & y_1 & y_2 & y_3 & y_4 & y_5 & y_6 & y_7
%\end{array}
%$$
%Therefore, in general, $|\x'\sqcap \y|= |\x\sqcap \y|+1$ and, in each step, we reduce the distance between the two vertices in one unity, and hence
%$$
%\dist(\x,\y)\le \ell-|\x\sqcap \y|=2\ell-2.
%$$
Finally, the existence of two vertices $\x$ and $\y$ at maximum distance is as follows. We have two cases:\\
If $\ell$ is even, consider the vertices $\x=1010\ldots1012$ and $\y=0202\ldots02$.\\
If $\ell$ is odd, consider the vertices $\x=0101\ldots012$ and $\y=0202\ldots021$.\\
Then, in both cases it is easily checked that $|\tilde{\x}\sqcap \y|=1$ and, hence, $\dist(\x,\y)=2\ell-2$.
\end{proof}

Fiol, Yebra, and Alegre~\cite{fya84} proved that if the diameter of any digraph (different from a directed cycle) is $D$, then the diameter of its line digraph is $D+1$.
Since $CK(d,\ell)$ are the line digraphs of the subKautz digraphs $sK(d,\ell-1)$, the diameter of the former is one unit more than the latter.

\begin{corollary}
\label{coro0}
The diameter of the subKautz digraph $sK(d,\ell)$ with $d\ge 4$ and $\ell\ge 3$ is $2\ell-1$.
\end{corollary}

\subsection{Routing and distances when $d=3$ or $\ell=3$}

\begin{figure}[t]
%\vskip-.75cm
    \begin{center}
  \includegraphics[width=12cm]{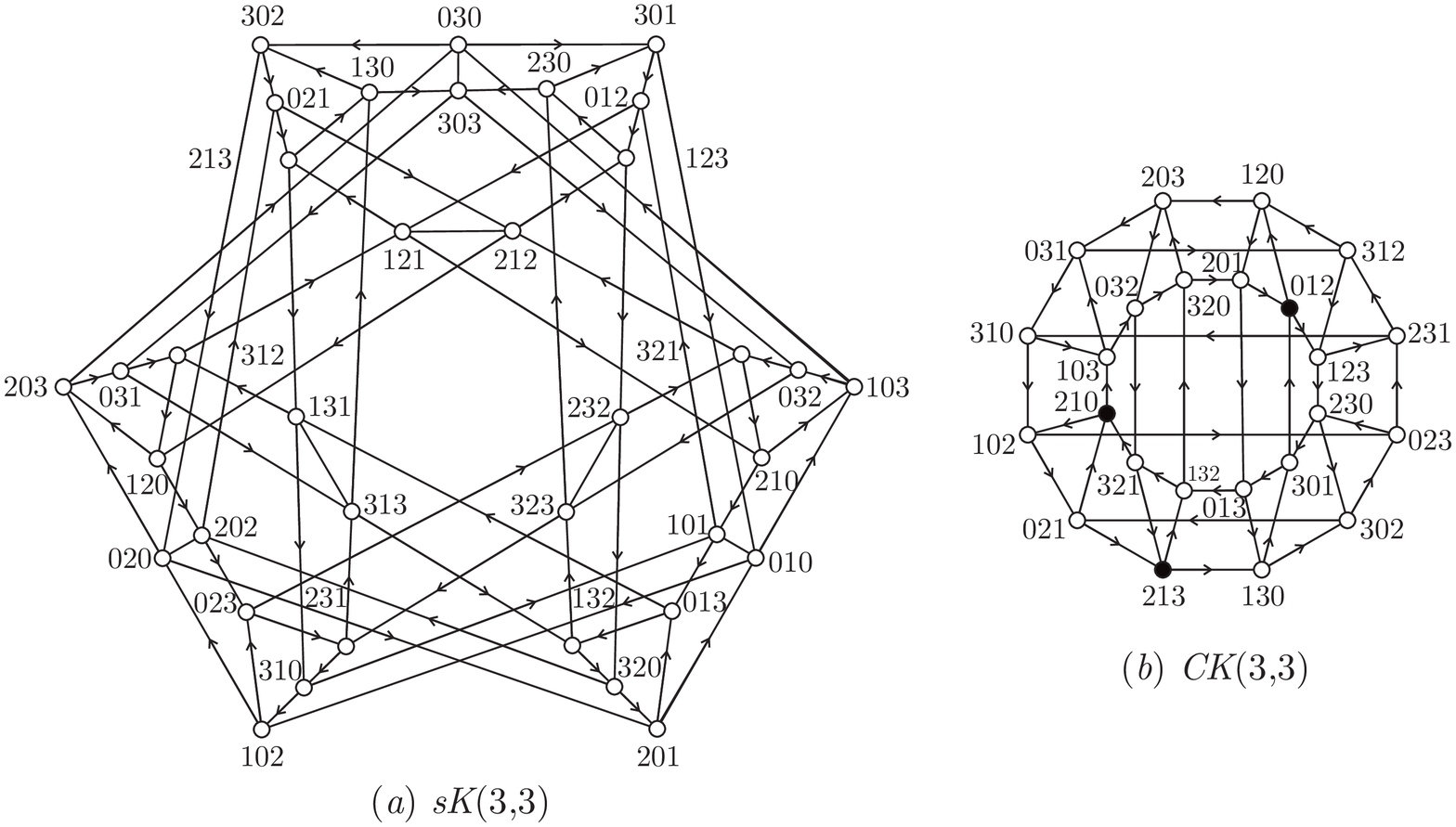}
  \end{center}
  %\vskip-2.75cm
  \caption{$(a)$ The subKautz digraph $sK(3,3)$ whose line digraph is $CK(3,4)$ (the lines without direction represent two arcs with opposite directions). $(b)$ The cyclic Kautz digraph $CK(3,3)$ with 24 vertices and diameter 5 (the vertices at maximum distance from 012 are 210 and 213).}
  \label{fig:sK(3,3)+CK(3,3)}
\end{figure}

Looking at the case $(c3)$ above, if $d=3$ and all the elements $z_2, x_4, y_6, y_1$ are different, then $z_3$ has no possible value.
Analogously, if $\ell=3$, there must exist two vertices $\x=x_1x_2x_3$ and $\y=y_1y_2y_3$, such that $|\tilde{\x}\sqcap \y|=2$ (not smaller than $\ell-1$), and with $y_1=x_3$. Thus, neither of the strategies  in the above cases $(c)$ and $(b)$ can be applied. However, the following reasoning shows that we always can find a path of length $2\ell-1$. First, we deal with the case $d=3$, where for simplicity we assume that $\ell=5$.

\begin{enumerate}
\item[$(d)$] We reason as if $|\tilde{\x}\sqcap \y|=0$:\\
%With intersection $-1$, we mean that the two sequences $\tilde{\x}$ and $\y$ are two steps apart.
$$
\begin{array}{cccccccccccccc}
x_1 & x_2 & x_3 & x_4 & x_5 & \overline{x_2} & \overline{x_3} & \overline{x_4} & \overline{x_5} &     &     &     &     &     \\
    &     &     &     &     & \overline{y_1} & \overline{y_2} & \overline{y_3} & \overline{y_4} & y_1 & y_2 & y_3 & y_4 & y_5 \\
    &     &     &     & =   & z_1            & z_2            & z_3            & z_4            & y_1 & y_2 & y_3 & y_4 & y_5
\end{array}
$$
where we would need the following conditions:\\
\indent$(d1)$ $z_1\neq x_2,x_5,y,1$,\\
\indent$(d2)$ $z_2\neq z_1,x_3,y_2$,\\
\indent$(d3)$ $z_3\neq z_2,x_4,y_3$,\\
\indent$(d4)$ $z_4\neq z_3,x_5,y_4,y-1$.\\
If $d\geq4$ (for $\ell=3$), this conditions can always be fulfilled, and the required path is guaranteed.

If $d=3$, and either $y_1=x_2$, or $y_2=x_3$, or $y_3=x_4$, or $y_4=x_5$, or $y_1=x_5$, then there is always a possible choice of $z_1,z_2,z_3$ and $z_4$ in $\Z_4$.
Consequently, $\dist(\x,\y)\leq9$. Otherwise, if $y_i\neq x_{i+1}$ for $i=1,\ldots,4$ and $y_1\neq x_5$, we can reason as if $|\tilde{\x}\sqcap \y|=4(=\ell-1)$. In this case,
the path from $\x$ to $\y$ is:
$$
\x=x_1 x_2 x_3 x_4 x_5 \ \rightarrow\  x_2 x_3 x_4 x_5y_1\ \rightarrow\  x_3 x_4 x_5y_1y_2 \ \rightarrow\  \cdots \ \rightarrow\  y_1 y_2 y_3 y_4 y_5=\y,
$$
which implies that $\dist(\x,\y)\le 5$.

Thus, in any case,
$$
\dist(\x,\y)\le 2\ell-1.
$$
\end{enumerate}
%Since there exists this path for any pair of vertices, $CK(d,\ell)$ is $(2\ell-1)$-reachable.

This leads to the following result.

\begin{proposition}
\label{propo1}
\indent$(i)$ The diameter of the cyclic Kautz digraphs $CK(3,\ell)$ with $\ell\neq 4$ and that of $CK(d,\ell)$ with $\ell=3$ is $2\ell-1$.\\
\indent$(ii)$ The diameter of the cyclic Kautz digraph  $CK(3,4)$ is $2\ell-2=6$.
\end{proposition}

\begin{proof}
\indent$(i)$ We only need to exhibit two vertices at distance  $2\ell-1$. For $CK(3,\ell)$ with $\ell\geq5$,
when $\ell$ is odd, we can take the vertices $\x={0 1 0 1 {\ldots}0 1 2}$ and $\y = {2 1 0 1 0 {\ldots}  1 0}$. When $\ell$ is even, two vertices at maximum distance are
$\x={1 0 2 0 2 0 {\ldots} 2 0 1 2}$ and $\y={2 1 3 0 2 0 2 {\ldots} 0 2 0 1 0}$. In both cases,
it was proved that these vertices are at maximum distance in \cite{bdh16}.
The case of the cyclic Kautz digraph $CK(3,3)$, shown in Figure \ref{fig:sK(3,3)+CK(3,3)} $(b)$, can be easily checked to have diameter $2\ell-1=5$, for instance, the vertices
at maximum distance from $012$ are $210$ and $213$. In general,
for $CK(d,3)$, we show that two vertices at maximum distance 5 are $\x=x_1 x_2 x_3$ and $\y=x_3 x_2 y_3$ as follows.
If this distance were 2, then we would get the sequence $x_1 x_2 x_3x_2y_3$, but $x_2 x_3x_2$ is not a vertex of $CK(d,3)$.
If this distance were 3, then we would get the sequence $x_1 x_2 x_3x_3x_2y_3$, but $x_2 x_3x_3$ is not a vertex of $CK(d,3)$.
If this distance were 4, then we would get the sequence $x_1 x_2 x_3y_1x_3x_2y_3$, but $x_3 y_1x_3$ is not a vertex of $CK(d,3)$.
Then, the distance is 5, with the sequence $x_1 x_2 x_3y_1y_2x_3x_2y_3$.\\
\indent$(ii)$ The cyclic Kautz digraph $CK(3,4)$ on 84 vertices with labels $x_1x_2x_3x_4$, $x_i\in\Z_4$, is the line digraph of the subKautz digraph $sK(3,3)$ shown in Figure \ref{fig:sK(3,3)+CK(3,3)} $(a)$.
Then, since $sK(3,3)$ has diameter $5$, we conclude that $CK(3,4)$ has diameter $6$, as claimed.
%Then the distance from $\x=0312$ to $\y=3212$ (with $|\tilde{\x}\sqcap \y|= 1$) is $2\ell-1=7$ since the possible shortest path of length $2\ell-2=6$, following the pattern in the above case {\bf $(b)$}, should have $z_1=0$, and $z_2\not\in\{z_1,x_3,y_1,y_2\}=\Z_4$, a contradiction.
\end{proof}

\begin{corollary}
\label{coro1}
\indent$(i)$ The diameter of the subKautz digraphs $sK(d,\ell)$ with either $d=3$ and $\ell\ge 4$ or $d\ge 3$ and $\ell=2$ is $2\ell$.\\
\indent$(ii)$ The diameter of the subKautz digraph $sK(3,3)$ is $2\ell-1=5$.
\end{corollary}

See Figure~\ref{fig:diametre-sK-CK} for a summary of the diameters of $sK(d,\ell)$ and $CK(d,\ell)$.

\begin{figure}[t]
%\vskip-0.75cm
\begin{center}
  \includegraphics[width=14.5cm]{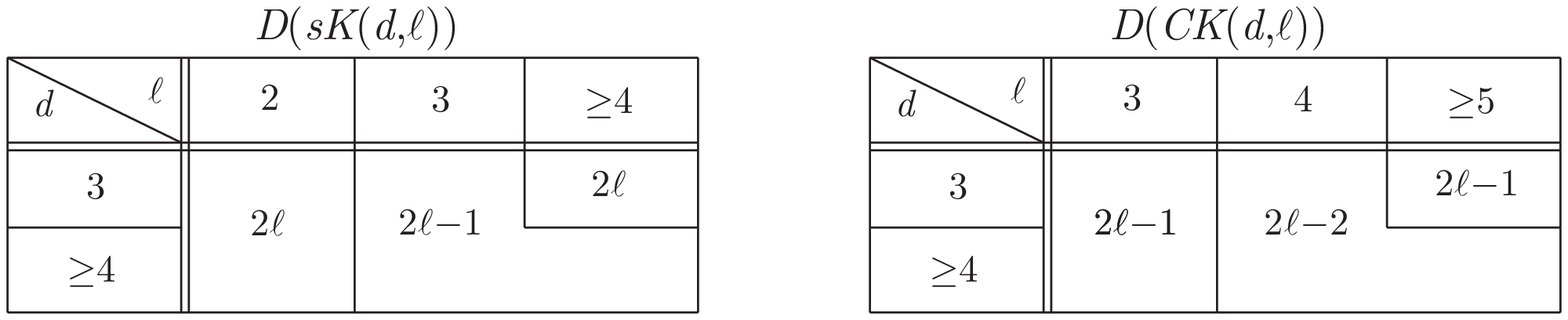}
  %\vskip-15.75cm
  \caption{Summary of the diameters of $sK(d,\ell)$ and $CK(d,\ell)$, depending on the values of $d$ and $\ell$.}
  \label{fig:diametre-sK-CK}
  \end{center}
\end{figure}

\subsection{The girth}

Now we give a lower bound on the girth of a cyclic Kautz digraph $CK(d,\ell)$.

\begin{lemma}
\label{lemma:girth}
The girth $g$ of the cyclic Kautz digraph $CK(d,\ell)$ is at least the minimum positive integer $k$ such that $\ell$ is not congruent with $1\ (\emph{mod}\ k)$.
\end{lemma}

\begin{proof}
A cycle of minimum length $g$, rooted to a vertex $\x$, corresponds to a path from $\x$ to
$\x$ of the same length. This means that the maximum  length of the (nontrivial) intersection $\x\sqcap\x$ is $\ell-g$. For instance, with $\ell=7$ and $g=4$ we would have the intersection pattern
$$
\begin{array}{ccccccccccccc}
x_1 & x_2 & x_3 & x_4 & x_5 & x_6 & x_7 & \overline{x_2} & \overline{x_3} & \overline{x_4} & \overline{x_5}& \overline{x_6} & \overline{x_7}\\
    &     &     &     & x_1 & x_2 & x_3 & x_4            & x_5            & x_6            & x_7.          &                &
\end{array}
$$
Then, in general, this means that the sequence representing $\x$ is periodic:
$x_i=x_{i+g}$ for every $i=1,2,\ldots,\ell-g$.
Now, if $\ell\equiv r\ (\textrm{mod}\ g)$, then $x_{\ell}=x_r$, which is possible if $r\neq 1$, and in this case the cycle would be
\begin{eqnarray*}
\x &= & x_1x_2\ldots x_g\ldots x_1x_2\ldots x_g x_1x_2\ldots x_{r}\\
  &\rightarrow& x_2\ldots x_g\ldots x_1x_2\ldots x_g x_1x_2\ldots x_{r} x_{r+1}\\
   &\rightarrow& \cdots \ \rightarrow \  x_{g-r+1}\ldots x_g \ldots x_1x_2\ldots x_g x_1x_2\ldots x_{r}x_{r+1}\ldots x_{g}\\
   &\rightarrow& x_{g-r+2}\ldots x_g \ldots x_1x_2\ldots x_g x_1x_2\ldots x_{r}x_{r+1}\ldots x_{g}x_1\\
   &\rightarrow& \cdots \ \rightarrow \  x_1x_2\ldots x_{g}\ldots x_1x_2\ldots x_{r}x_{r+1}\ldots x_{g}x_1\ldots x_r=\x.
%\x &=x_1x_2x_3\ldots x_{r}\ \rightarrow\ x_2x_3\ldots x_{r} x_{r+1}\ \rightarrow\ \cdots \rightarrow\ x_{g-r+1}\ldots x_{r}x_{r+1}\ldots x_{g}\\
% & \ \ \ \rightarrow\ x_{g-r+2}\ldots x_{r}x_{r+1}\ldots x_{g}x_1 \rightarrow\
% \cdots \rightarrow\ x_1\ldots x_{r}x_{r+1}\ldots x_{g}x_1\ldots x_r=\x.
\end{eqnarray*}
This completes the proof.
\end{proof}

Note that the girth reaches the bound when there exists a vertex $\x$ that satisfies the cases $(a)$, $(b)$, $(c)$ or $(d)$ (given at the beginning
of this section) for the existence of a path of length $g$ from $\x$ to $\y=\x$. In particular, this is fulfilled if $d$ is large enough.
As an example, if $\ell=13$ Lemma~\ref{lemma:girth} gives $g\geq5$. However, a possible vertex $\x$ only exists for $d\geq4$. Indeed, assume
that $\x=x_1x_2x_3x_4x_5x_1x_2x_3x_4x_5x_1x_2x_3$, where $x_i\in\Z_4$ for $i=1,\ldots,5$. Since $x_2\neq x_1$ and $x_3\neq x_2,x_1$,
we can take, without loss of generality $\x=012x_4x_5012x_4x_5 012$. Then, a path of length $g=5$ from $\x$ to $\x$ should be
\begin{eqnarray*}
\x&=&012x_4x_5012x_4x_5 012 \ \rightarrow\  12x_4x_5012x_4x_5 012x_4 \ \rightarrow\  2x_4x_5012x_4x_5 012x_4x_5\\
&&x_4x_5012x_4x_5 012x_4x_5 0 \ \rightarrow\  x_5012x_4x_5 012x_4x_5 01 \ \rightarrow\  012x_4x_5012x_4x_5 012\ =\ \x.
\end{eqnarray*}
Therefore, since $x_4\neq2,1$ and $0\neq x_4$, then $x_4=3$. Moreover, since $x_5\neq x_4,2$, $0\neq x_5$, and $1\neq x_5$, then $x_5\not\in\{0,1,2,3\}$, which is a contradiction.
In fact, when $d=3$, it turns out that $CK(3,13)$ has girth $g=7$, for example, with the vertex $\x=0120123012012$.

A direct consequence of this result is that there exist cyclic Kautz digraphs with arbitrarily large girth. Indeed, if $\ell=\textrm{lcm}(2,3,\ldots,n)+1$, we have that
$\ell=1\ (\textrm{mod}\ i)$ for every $i=2,3,\ldots,n$. Then, according to Lemma~\ref{lemma:girth}, $CK(d,\ell)$ must have girth $g>n$.

It is known that if a digraph $G$ has girth $g$, then its line digraph $L(G)$ also has girth $g$, see F\`{a}brega and Fiol~\cite{FaFi89}. Since $L(sK(d,\ell))=CK(d,\ell+1)$,
both digraphs have the same girth.

\section{Connectivity and superconnectivity}

It is well-known that the Kautz digraphs $K(d,\ell)$  have maximal (edge- and vertex-) connectivities (see F\`abrega and Fiol \cite{FaFi89}). The following result shows that this is also the case for the other Kautz-like digraph studied here, see Figure~\ref{fig:taula-sK-CK} for a summary.

%\begin{proposition}
%$(i)$ The subKautz digraph $sK(d,\ell)$ with either $d=3$ and $\ell\neq 3$, or $d\ge 4$ and $\ell=2$, is super-$\lambda$.\\
%$(ii)$ The subKautz digraph $sK(d,\ell)$ with either $d=\ell=3$, or $d\ge 4$ and $\ell\ge 3$, is super-$\lambda$ and maximally vertex-connected.\\
%$(iii)$ The cyclic Kautz digraph $CK(d,\ell)$ with either $d=3$ and $\ell\neq 4$, or $d\ge 4$ and $\ell=3$, is super-$\lambda$ and maximally vertex-connected.\\
%$(iv)$ The  cyclic Kautz digraph $CK(d,\ell)$ whit either $d=3$ and $\ell=4$, or $d,\ell\ge 4$, is super-$\lambda$ and super-$\kappa$.
%\end{proposition}

\begin{proposition}
\indent$(i)$ The subKautz digraph $sK(d,\ell)$ with $d\geq3$ and $\ell\geq2$ is super-$\lambda$.\\
\indent$(ii)$ The subKautz digraph $sK(d,\ell)$ with either $d=\ell=3$, or $d\geq4$ and $\ell\geq3$, is maximally vertex-connected.\\
\indent$(iii)$ The cyclic Kautz digraph $CK(d,\ell)$ with $d\geq3$ and $\ell\geq3$ is super-$\lambda$.\\
\indent$(iv)$ The cyclic Kautz digraph $CK(d,\ell)$ with either $d=3$ and $\ell=4$, or $d,\ell\geq4$, is super-$\kappa$.\\
\indent$(iv)$ The cyclic Kautz digraph $CK(d,\ell)$ with either $d=3$ and $\ell\neq4$, or $d\geq4$ and $\ell=3$, is maximally vertex-connected.
\end{proposition}

\begin{proof}
Since both $sK(d,\ell)$ and $CK(d,\ell)$ are subdigraphs of $K(d,\ell)$, with semigirth $\ell$ (see F\`abrega and Fiol \cite{FaFi89}), then the semigirths of these digraphs are at least $\ell$.
Hence, by using that the diameters of $sK(d,\ell)$ and $CK(d,\ell)$ are given in Theorem \ref{theo1}, Proposition \ref{propo1}, and Corollaries \ref{coro0} and \ref{coro1}, the result follows from Theorems~\ref{theo1FaFi} and~\ref{theo2FaFi}.
\end{proof}

\begin{figure}[t]
%\vskip-0.75cm
\begin{center}
  \includegraphics[width=14.5cm]{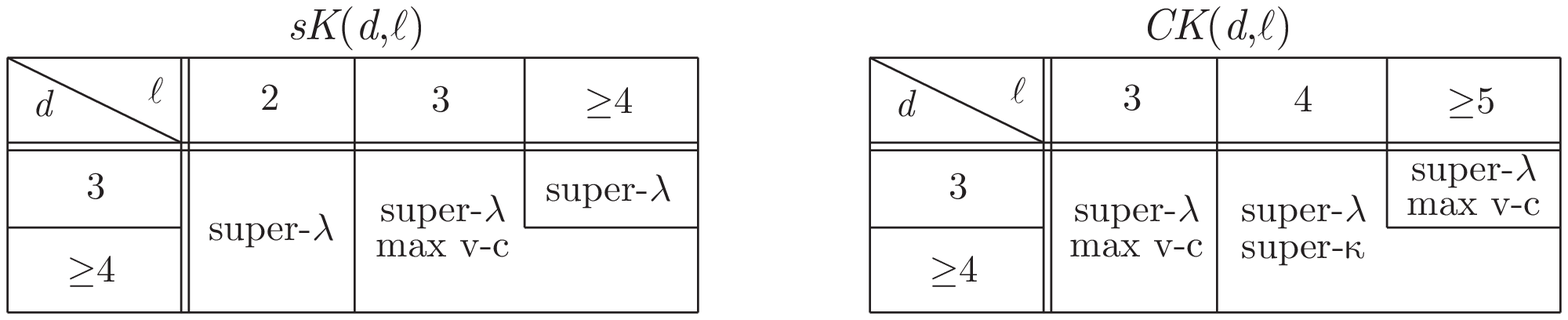}
  %\vskip-15.75cm
  \caption{Summary of the connectivities of $sK(d,\ell)$ and $CK(d,\ell)$, depending on the values of $d$ and $\ell$.}
  \label{fig:taula-sK-CK}
  \end{center}
\end{figure}

\section{Cyclic Kautz digraphs $CK(d,3)$ with $d\geq3$}

The cyclic Kautz digraphs $CK(d,3)$ with $d\geq3$ have some special properties that, in general, are not shared with $CK(d,\ell)$ with $\ell>3$.
These properties are listed in the following result.

\begin{lemma}
The cyclic Kautz digraphs $CK(d,3)$ with $d\geq3$ satisfy the following properties:\\
\indent$(a)$ $(d-1)$-regular.\\
\indent$(b)$ Number of vertices: $N=d^3-d$, number of arcs: $m=(d+1)d(d-1)^2$.\\
\indent$(c)$ Diameter: $2\ell-1=5$.\\
\indent$(d)$ $CK(d,3)$ %with $d\geq3$
are the line digraphs of the subKautz digraphs $sK(d,2)$, which are obtained from the Kautz digraphs $K(d,2)$ by removing the arcs of the digons.\\
\indent$(e)$ Vertex-transitive.\\
\indent$(f)$ Eulerian and Hamiltonian.
\end{lemma}

\begin{proof}
$(a)$, $(b)$, $(c)$ and $(d)$ come from the properties of general $CK(d,\ell)$. $(e)$ Since $sK(d,2)$ (with $d\geq3$) are vertex-transitive and arc-transitive, their line digraphs $CK(d,3)$ are vertex-transitive.
$(f)$ $sK(d,2)$ and $CK(d,3)$ with $d\geq3$ are Eulerian, because they are $(d-1)$-regular. Since $sK(d,2)$ (with $d\geq3$) are Eulerian, their line digraphs $CK(d,3)$ are Hamiltonian.
\end{proof}

\subsection{Mean distance}

As said before, $CK(d,\ell)$ are asymptotically optimal with respect to the mean distance. Now, we give the exact formulas for the mean distance of $sK(d,2)$ and $CK(d,3)$ with $d\geq3$.
Let $n$ and $N$ be the numbers of vertices of $sK(d,2)$ and $CK(d,3)$, respectively.

\begin{lemma} $(a)$ The mean distance of the antipodal subKautz digraph $sK(d,2)$ with $d\ge 3$ is
\begin{equation}
\overline{\partial^*}=\frac{2d^2+3d-1}{d^2+d}.
\end{equation}
\indent$(b)$ The mean distance of the cyclic Kautz digraph $CK(d,3)$ with $d\geq3$ is
\begin{equation}
\overline{\partial}=\frac{3d^3+d^2-5d-2}{d^3-d}.
\end{equation}
\end{lemma}

\begin{proof}
Since $CK(d,3)$ (and also $sK(2,2)$) with $d\geq3$ is vertex-transitive, we can compute the number of vertices from any given vertex.
First, we fix the distance layers in $sK(2,2)$. Thus, in Table \ref{table2},
we give the numbers $n_k(u,v)$ of vertices at distance $k=0,1,\ldots,4$ from vertex $u=01$ to vertex $v\in\{01,1x,\ldots,10\}$.
%we  give the numbers $n_k(u,v)$ of vertices $v$ at distance $k=0,1,\ldots,4$ from vertex $u=01$.
\begin{table}[h]
\begin{center}
\begin{tabular}{|c|c|c|c|}
\hline
$u$ & $v$ & $k=\dist(u,v)$ & $n_k(u,v)$ \\[0.1cm]
\hline
$01$ & $01$ & 0 & $1$  \\
$01$ & $1x$ & 1 & $d-1$  \\
$01$ & $x0$ & 2 & $d-1$   \\
$01$ & $xy$ & 2 & $(d-1)(d-2)$   \\
$01$ & $x1$ & 3 & $d-1$   \\
$01$ & $0x$ & 3 & $d-1$   \\
$01$ & $10$ & 4 & $1$  \\
\hline
\end{tabular}
\end{center}
\caption{Numbers of vertices $v$ at distance $k$ from $u=01$.}
\label{table2}
\end{table}

Then, the total numbers $n_i=n_i(u)$ of vertices at distance $i=0,1,\ldots,4$ from $u$ turn out to be
$$
n_0 =1, \quad n_1 =d-1, \quad n_2 =(d-1)^2, \quad n_3 = 2(d-1), \quad n_4 = 1,
$$
with $n=n_0+n_1+\cdots+n_4=d^2+d$, and showing that $sK(2,2)$ is antipodal.

Now we use again that $CK(d,3)$ is the line digraph of $sK(d,2)$ to conclude that, in the former,
the numbers $N_i$ of vertices at distance $i=0,1,\ldots,5$ from a given vertex, say $201$, are
\begin{eqnarray*}
  && N_0 = n_0 =1, \quad N_1 = n_1 =d-1, \quad N_2 = (d-1)n_1 =(d-1)^2, \quad N_3 = (d-1)n_2-1\\
  && =(d-1)^3-1, \quad N_4 = (d-1)n_3 =2(d-1)^2, \quad N_5 = (d-1)n_4 =d-1,
\end{eqnarray*}
satisfying $N=N_0+N_1+\cdots+N_5=d^3-d$, as requested.

Note that in $N_3=(d-1)n_2-1$ we subtract one unit due to the presence in $sK(d,2)$ of the cycle of length 3: $20\rightarrow 01\rightarrow 12 \rightarrow 20$.
Then, the mean distances  of $CK(d,3)$ with $d\geq3$ are, respectively,
$\displaystyle\overline{\partial^*}=\frac{1}{n}\sum_{k=0}^4 kn_k$, and $\displaystyle\overline{\partial}=\frac{1}{N}\sum_{k=0}^5 kN_k$, which gives the results.
\end{proof}

Observe that, since $CK(d,3)$ is the line digraph of $sK(d,2)$, the respective mean distance satisfies the inequality $\overline{\delta}<\overline{\delta^*}$, in concordance with the results by Fiol, Yebra, and Alegre \cite{fya84}. Also, note that the mean distances of $sK(d,2)$ and $CK(d,3)$, with $d\geq3$, tend, respectively, to $2$ and $3$ for large degree $d-1$, that is, they are asymptotically optimal.

%\subsection*{Open problems}
%
%\begin{itemize}
%\item
%Quan es descriuen els camins de longitut m\'{\i}nima, s'ha de comentar que els nombres de vertexs a dist\`{a}ncia $\ell$ (cas d'intersecci\'{o} nul.la) \'{e}s de l'ordre de $d^{\ell}$, que correspon tamb\'{e} a l'ordre de la dist\`{a}ncia mitjana \`{o}ptima.
%\end{itemize}

\bibliographystyle{plain}

\end{document}